\documentclass[11pt]{amsart}
\usepackage{amssymb,amscd,amsmath}
\usepackage[all]{xy}
\newtheorem{thm}{Theorem}

\newtheorem{prop}[thm]{Proposition}

\newtheorem{lem}[thm]{Lemma}
\theoremstyle{remark}

\theoremstyle{definition}

\newcommand{\R}{{\mathbb{R}}}
\newcommand{\C}{{\mathbb{C}}}

\newcommand{\Z}{{\mathbb{Z}}}

\newcommand{\End}{\mathrm{End}}

\newcommand{\GL}{\mathrm{GL}}
\newcommand{\SL}{\mathrm{SL}}

\newcommand{\diag}{\mathrm{diag}}

\newcommand{\ad}{\mathrm{ad}}

\renewcommand{\sl}{\mathfrak{sl}}

\newcommand{\g}{\mathfrak g}

\newcommand{\frakt}{\mathfrak t}

\newcommand{\gl}{\mathfrak{gl}}

\newcommand{\Alt}{{\raise 2pt\hbox{$\scriptstyle\bigwedge$}}}

\begin{document}
\title{Disjointness of a simple matrix Lie group and its Lie algebra}

\author{Michael Larsen}
\email{mjlarsen@indiana.edu}
\address{Department of Mathematics\\
    Indiana University \\
    Bloomington, IN 47405\\
    U.S.A.}


\thanks{The author was partially supported by NSF grant DMS-2001349.}

\begin{abstract}
Let $G$ be a connected closed subgroup of $\GL_n(\C)$ which is simple as a Lie group and which acts irreducibly on $\C^n$.  Regarding both $G$ and its Lie algebra $\g$ as subsets of $M_n(\C)$,
we have $G\cap \g\neq\emptyset$ if and only if $G$ is a classical group and $\C^n$ is a minuscule representation.
\end{abstract}

\maketitle

\begin{center}
\textit{To Pham Huu Tiep on his sixtieth birthday, with admiration}
\end{center}
\vskip 10pt
Let $G$ be a simple complex Lie group with Lie algebra $\g$ and $\rho\colon G\to \GL(V)$ a faithful, irreducible, complex representation of $G$.  Let $\rho_*\colon \g\to \gl(V)$ denote the corresponding Lie algebra representation.
Using the embeddings $G\stackrel{\rho}{\hookrightarrow}\GL(V)\hookrightarrow \End(V)$ and $\g\stackrel{\rho_*}{\hookrightarrow}\gl(V) = \End(V)$, it makes
sense to ask if $G\cap \g$, by which we abbreviate $\rho(G)\cap \rho_*(\g)$, is empty.  As we will see, the answer is usually yes, but the exceptions form a class which has received attention before (see \cite{Serre}).

Recall \cite[VIII, \S7.3]{Bourbaki7} that $V$ is \emph{minuscule} if it is non-trivial and its weights with respect to a choice of maximal torus form a single orbit under the action of the Weyl group.  The minuscule representations are as follows \cite[Appendix]{Serre}:  for type $A_r$, all exterior powers of the natural representation; for $B_r$, the spin representation; for $C_r$, the natural representation; for $D_r$, the natural representation and the semispin representations; for $E_6$, the two $27$-dimensional irreducibles; and for $E_7$, the unique $56$-dimensional irreducible.

The main result of this paper is as follows:

\begin{thm}
\label{main}
Under the above hypotheses, $G\cap \g$ is non-empty if and only if $G$ is of classical type and $V$ is minuscule.
\end{thm}

The rest of the paper proves the theorem.  The main difficulty is to show necessity of the conditions on $(G,V)$.

Let $g\in G$ and $x\in \g$ satisfy $\rho(g) = \rho_*(x)$.
The Jordan decomposition $g = tu$ in $G$ gives the Jordan decomposition $\rho(g) = \rho(t)\rho(u)$ in $\GL(V)$.  In $M_n(\C)$, we can write $\rho(u) = 1+N$, where $N$ is nilpotent and commutes with $\rho(t)$,
so $\rho(t) +\rho(t)N$ is the Jordan-Chevalley decomposition of $\rho_*(x) = \rho(g)$ in $\End(V)$.  By \cite[Chap.~4, Lemma~A]{Hump}, it follows that $\rho(t)\in \rho_*(\g)$.  Without loss of generality, therefore,
we may assume that $g$ is semisimple, and redefining $x$ if necessary,
we still have $\rho(g) = \rho_*(x)$.  Let $T$ be a maximal torus of $G$ containing $g$, and let $\frakt$ denote its Lie algebra.  

Now, $\rho_*(x)$ commutes with $\rho(t)$ for all $t\in T$.  As $\rho_*$ is injective,
$\ad(t)(x) = x$ for all $t\in T$, so $x$ lies in the weight-$0$ space of the adjoint representation of $G$.
In other words, $x \in \frakt$.

As usual, we identify the character group $X^*(T)$ with a subgroup of $\frakt^*$ by means of the diagram
$$\xymatrix{\frakt\ar[r]^{\exp}\ar[d]_{t^*}&T\ar[d]^\chi\\
\C\ar[r]^\exp&\C^\times}$$
relating a character $\chi$ and its corresponding vector $t^*$.
We choose a basis $e_1,\ldots,e_n$ of $V$ consisting of weight vectors for $T$ for characters $\chi_1,\ldots,\chi_n\in X^*(T)$ and let $t^*_1,\ldots,t^*_n$ denote the elements of $\frakt^*$ corresponding to $\chi_1,\ldots,\chi_n$ respectively.
If $a_1,\ldots,a_n\in\Z$ and $\chi_1^{a_1}\cdots \chi_n^{a_n} = 1$, then $a_1t^*_1+\cdots+a_nt^*_n=0$.

\begin{lem}
\label{AP}
If $\chi_i,\chi_j,\chi_k$ are three characters of $V$ such that $\chi_i\chi_k = \chi_j^2$, then $\chi_i(g)=\chi_j(g)=\chi_k(g)$.
\end{lem}

\begin{proof}
As $\chi_i \chi_j^{-2} \chi_k = 1$, we have $t^*_i-2t^*_j+t^*_k=0$.  As $\rho(g) = \rho_*(x)$, we have
$$\chi_i(g) = t^*_i(x),\,\chi_j(g) = t^*_j(x),\,\chi_k(g) = t^*_k(x),$$
so we have
$$4\chi_i(g)\chi_k(g) = 4\chi_j(g)^2 = 4t^*_j(x)^2 = (t^*_i(x)+t^*_k(x))^2 = (\chi_i(g) + \chi_k(g))^2,$$
so $\chi_i(g) = \chi_k(g)$.  Thus,
$$\chi_j(g) = t^*_j(x) = \frac{t^*_i(x)+t^*_k(x)}2 = \frac{\chi_i(g)+\chi_k(g)}2 = \chi_i(g).$$
\end{proof}

\begin{lem}
\label{quadrangle}
If $\chi_i\chi_j = \chi_k\chi_l$ for characters of $V$, then 
\begin{equation}
\label{quadratic}
(z-\chi_i(g))(z-\chi_j(g)) = (z-\chi_k(g))(z-\chi_l(g)).
\end{equation}
\end{lem}

\begin{proof}
The character equality implies
\begin{equation}
\label{equal sum}
t^*_i+t^*_j = t^*_k+t^*_l,
\end{equation}
so
$$\chi_i(g)+\chi_j(g) = t^*_i(x)+t^*_j(x) = t^*_k(x) + t^*_l(x) = \chi_k(g)+\chi_l(g).$$
As $\chi_i(g)\chi_j(g) = \chi_k(g)\chi_l(g)$, equation \eqref{quadratic} follows immediately.
\end{proof}

\begin{lem}
\label{self-dual}
If $V$ is self-dual, then $\chi_i(g)^2+1=0$ for all $\chi_i$.
\end{lem}

\begin{proof}
Let $\chi_i\chi_j=1$.  Then $t^*_i+t^*_j=0$, and
$$\chi_i(g) = t^*_i(x) = -t^*_j(x) = -\chi_j(g) = -\chi_i(g)^{-1}.$$
Thus $\chi_i(g)^2+1=0$.
\end{proof}

\begin{prop}
If $V$ is not a minuscule representation, then $G\cap \g=\emptyset$.
\end{prop}

\begin{proof}
We fix a maximal torus and a Weyl chamber.
Let $\varpi_1,\ldots,\varpi_r$ denote the fundamental characters, and let $\lambda = \sum_i a_i \varpi_i$ denote the dominant weight of $V$.
Let $\alpha$ denote the highest root in the dual root system.  Let $s = \langle \lambda,\alpha^\vee\rangle$.  By \cite[\S21.3]{Hump}, the length of the $\alpha^\vee$-string of weights of $V$  through $\lambda$ is $s+1$.  If $s\ge 2$, then by Lemma~\ref{AP}, $\alpha^\vee(g) = 1$.  By symmetry,
the same is true for all short roots.  Since the root lattice is generated by short roots, it is true for character in the root lattice, so  $\rho(g)$ is an invertible scalar matrix.  However, $\rho_*(x)$ lies in $\sl(V)$ and therefore has trace $0$,
so this is impossible.  Therefore, $s\le 1$.  Now, $\langle \varpi_i,\alpha^\vee\rangle \ge 1$ with equality if and only if $\varpi_i$ is a minuscule weight \cite[Appendix]{Serre}.  Therefore the proposition holds whenever $\lambda$ is a fundamental weight.
On the other hand, if $\sum_i a_i\ge 2$, then
$$\langle \lambda,\alpha^\vee\rangle  = \sum_i a_i \langle \varpi_i,\alpha^\vee\rangle \ge \sum_i a_i \ge 2,$$
and the proposition again holds.
\end{proof}

\begin{prop}
If $G$ is of type $E_6$ and $\dim V = 27$, then $G\cap \g=\emptyset$.
\end{prop}

\begin{proof}
%
If $\{t^*_i(x)\mid 1\le i\le 27\}$ has only $1$ element, then since the set of differences $t_i^*-t_j^*$ span $\frakt^*$, it follows that $x=0$.  This implies $\rho(g)=0$, which is impossible.
If $\{t^*_i(x)\mid 1\le i\le 27\}$ has exactly $2$ elements, then there exists an affine linear function on $\frakt$ which takes the value $0$ and $1$ on every element $t^*_i$ and therefore 
a value in $[0,1]$ on every point of the convex hull of $\{t^*_i\mid 1\le i\le 27\}$ in $X^*(T)\otimes\R\subset \frakt^*$.
This convex hull is the polytope denoted $2_{21}$ in Coxeter's notation, and each of its codimension $1$ faces has $6$ or $10$ vertices \cite[Example 8.5.16]{Green}, so this is not possible.
Therefore, $\{t^*_i(x)\mid 1\le i\le 27\}$ has at least $3$ values.

Endowing $X^*(T)\otimes \R$ with the Weyl-invariant inner product for which the roots have length $\sqrt 2$, we see that the inner product of any weight with any element of the root lattice is integral,
while the inner product of any weight of $V$ with itself is $4/3$ \cite[Planches]{Bourbaki4}.  Therefore, the inner product $\langle t^*_i,t^*_j\rangle$ of any two distinct weights of $V$ is congruent to $1/3$ (mod $1$) and contained in the interval $[-4/3,4/3)$.
The only possible values are therefore $1/3$ and $-2/3$.  In the former case, we say that $t^*_i$ and $t^*_j$ are \emph{skew}.  In the latter case, we say they are \emph{incident}.
(The terminology is motivated by the correspondence
\cite[Example~8.6.4]{Green} between the weights of $V$ and the $27$-lines on a cubic surface.)  For each weight, the number of incident weights is $10$ \cite[Lemma 10.1.6]{Green}.

We claim that Lemma~\ref{quadrangle} implies that $\{t^*_1(x),\ldots,t^*_{27}(x)\}$ has at most $2$ elements.
We prove this by contradiction, assuming that $t^*_i(x)$, $t^*_j(x)$, and $t^*_k(x)$ are pairwise distinct.  There are four possibilities regarding $\{t^*_i,t^*_j,t^*_k\}$ to consider.

\vskip 10pt\noindent
\textbf{Case 1.  No pair of the weights is incident}.  By \cite[Exercise 8.1.11]{Green} and \cite[Lemma 9.2.7]{Green}, the Weyl group of $E_6$ acts transitively on triples of pairwise skew weights.
By \cite[Lemma 10.1.9]{Green}, we may assume without loss of generality that $t^*_i$, $t^*_j$, and $t^*_k$ belong to a $6$-element set of pairwise skew weights and therefore, by \cite[Theorem 10.2.1]{Green}, to one half of a Schl\"afli double six.
In the other half, there are three weights which are linked to each of $t^*_i$, $t^*_j$, and $t^*_k$; choose one of them and denote it $t^*_l$.  
We claim that $-t^*_m-t^*_i$, $-t^*_m-t^*_j$, and $-t^*_m-t^*_k$ are all weights of $V$.  Indeed, they lie in the dual lattice to the root lattice of $E_6$, and each has the same length as a weight of $V$.  This implies that each is either a weight or the negative of a weight  \cite[Chapter 4, (122)]{CS}.  They are actually weights since their inner products with every weight is congruent to $1/3$ (mod $1$) rather than $2/3$.
However,
$$(-t^*_m-t^*_i) + t^*_i = (-t^*_m-t^*_j)+t^*_j = (-t^*_m-t^*_k) + t^*_k.$$
By Lemma~\ref{quadrangle}, this implies that $t^*_i(x)$, $t^*_j(x)$, and $t^*_k(x)$ are all roots of a common quadratic polynomial, which is absurd.

\vskip 10pt\noindent
\textbf{Case 2.  Exactly one pair of the weights is incident}.  We may assume $t^*_i$ and $t^*_j$ are incident.  We have
$$(t^*_i+t^*_j-t^*_k)^2 = \frac 43,$$
By the same reasoning as in Case 1, $t^*_i+t^*_j-t^*_k$ is a weight $t^*_l$,  so \eqref{equal sum} holds.  However, $t^*_k(x)$ is not a root of $(z-t^*_i(x))(z-t^*_j(x))$, contrary to assumption.

\vskip 10pt\noindent
\textbf{Case 3.  Exactly two pairs of the weights are incident}.  Assume that $t^*_i$ and $t^*_j$ are skew, and both are incident to $t^*_k$.
By \cite[Proposition 10.2.7(iii)]{Green},
the number of weights incident to both $t^*_i$ and $t^*_j$ (including $t^*_k$) is $5$, so the number of weights incident to at least one of the two is $15$, including $t^*_k$.
Therefore, among weights not in $\{t^*_i,t^*_j,t^*_k\}$, the number skew to both $t^*_i$ and $t^*_j$ is $10$, and the number incident to $t^*_k$ is $8$.  We conclude that there exists $t^*_m$ which is skew to all three.
Therefore, we may choose two elements of $\{t^*_i,t^*_j,t^*_k\}$ together with $t^*_m$ and obtain three weights, all taking different values in $x$, containing at most a single incident pair.  This is impossible by Cases 1 and 2.

\vskip 10pt\noindent
\textbf{Case 4.  All pairs of the weights are incident}.  The number of incident pairs with one element in $\{t^*_i,t^*_j,t^*_k\}$ and one element in its complement is $24$, so on average each weight in the complement is incident to one element of 
$\{t^*_i,t^*_j,t^*_k\}$.  Let $t^*_m$ be a weight incident to at most one element of $\{t^*_i,t^*_j,t^*_k\}$.  Without loss of generality, we may assume that $t^*_i$, $t^*_j$, and $t^*_m$ all take different values on $x$, and there are at most two incident pairs among them.  This is impossible by Cases 1--3.

\vskip 10pt\noindent
This proves the claim and therefore the proposition.

%
\end{proof}

\begin{prop}
If $G$ is of type $E_7$ and $\dim V = 56$, then $G\cap \g=\emptyset$.
\end{prop}

\begin{proof}
There is a unique irreducible $56$-dimensional representation of $E_7$, and it is therefore self-dual.
By Lemma~\ref{self-dual}, we may assume that $t^*_j(x) = \chi_j(g) \in \{\pm i\}$ for $1\le j\le 56$.
This implies that the set of $t^*_j$ is a union of two hyperplanes, which must therefore bound the convex hull of $\{t^*_j\mid 1\le j\le 56\}$, which is the convex polytope $3_{21}$ in Coxeter's notation.
However, the codimension $1$ faces of $3_{21}$ each have $7$ or $12$ vertices \cite[Exercise 8.5.17]{Green}, so this is impossible.
\end{proof}

This concludes the proof of the only if direction of Theorem~\ref{main}.

\begin{lem}
Let $m\ge 2$ and $j\in [1,m-1]$ be integers, and let $V = \wedge^j \C^m$.  Let $G$ denote the image of $\SL_m(\C)$ in $\GL_n(V)$ via the exterior $j$th power map.  Then $G\cap \g\neq\emptyset$.
\end{lem}

\begin{proof}
Choose $a\in \C$ such that $a^m = \frac j{j-m}$, and let $b = a^{1-m}$.  Let $e_1,\ldots,e_m$ denote the standard basis of $\C^m$.
We fix an ordered basis of $V$ consisting of the vectors of the form $e_{i_1}\wedge\cdots\wedge e_{i_j}$, where $i_1>\cdots>i_j$, taken in lexicographic order.
Thus, the image of $\diag(a,\ldots,a,a^{1-m})\in \SL_m(\C)$ in $\GL(V)$ is
$$\diag(\underbrace{a^j,\ldots,a^j}_{\binom{m-1}j},\underbrace{a^{j-m},\ldots,a^{j-m}}_{\binom{m-1}{j-1}})=\frac{a^j}j\diag(\underbrace{j,\ldots,j}_{\binom{m-1}j},\underbrace{j-m,\ldots,j-m}_{\binom{m-1}{j-1}})\in\g.$$
\end{proof}

\begin{lem}
Let $V$ be a representation of $G$, $T$ a maximal torus of $G$, and $\frakt$ the Lie algebra of $T$.
If $x\in \frakt$ satisfies $t^*_j(x) \in \{\pm i\}$ for all $j$, then $\rho_*(x)\in G\cap \g$.
\end{lem}

\begin{proof}
Choose a basis of $V$ whose $j$th vector belongs to the ($1$-dimensional) $\chi_j$-weight space of $V$.  In terms of this basis, $\rho_*(x)$ is given by the matrix
$\diag(t^*_1(x),\ldots,t^*_n(x))$.  Inside the group $D$ of invertible diagonal matrices, the image $\rho(T)$ is a closed subgroup \cite[Corollary 1.4(a)]{Borel} and it is diagonal, so it is the 
intersection of the kernels of characters of $D$ which vanish on it \cite[Proposition 8.2(c)]{Borel}.  An element of $X^*(D)$ vanishes on $\rho(T)$ if and only if it lies in the kernel of the natural homomorphism
$X^*(D)\to X^*(T)$.  Thus, if $A$ is the set of $n$-tuples $(a_1,\ldots,a_n)\in \Z^n$ such that $\chi_1^{a_1}\cdots\chi_n^{a_n}=1$, then
$$\rho(T) = \{\diag(c_1,\ldots,c_n)\in D\mid c_1^{a_1}\cdots c_n^{a_n}=1\,\forall (a_1,\ldots,a_n)\in A\}.$$

For all $(a_1,\ldots,a_n)\in A$, we have $\sum_j a_jt^*_j = 0$, so $\sum_j a_jt^*_j(x) = 0$.  Writing $t^*_j(x) = \epsilon_j i$, we have $\sum_j a_j \epsilon_j=0$, so
$$\prod_j t^*_j(x)^{a_j} = \prod_j (\epsilon_j i)^{a_j} = \prod_j (i^{\epsilon_j})^{a_j} = i^{\sum_j a_j\epsilon_j} = 1.$$
This implies $\rho_*(x)\in \rho(T)$.

\end{proof}

For minuscule representations when $G$ is of type $B$, $C$, or $D$, such an $x$ always exists; it suffices to find an element of $X^*(T)\otimes \R$ whose inner product with each weight of $V$ is $\pm 1$.
Using the notation of \cite[Planches]{Bourbaki4}, these can be chosen as follows.  For type $B_r$ and highest weight $\varpi_r$, we take $(1,0,\ldots,0)$.  For type $C_r$ and highest weight $\varpi_1$, we take $(1/2,0,\ldots,0)$.
For type $D_r$ and highest weight $\varpi_1$, we take $(1/2,0,\ldots,0)$.  For type $D_r$ and highest weight $\varpi_{r-1}$ or $\varpi_r$, we take $(1,0,\ldots,0)$.

\end{document}